\documentclass[11pt]{article}

\usepackage[T1]{fontenc}
\usepackage[utf8]{inputenc}
\usepackage{lmodern}
\usepackage[margin=1in]{geometry}
\usepackage{microtype}

\usepackage{amsmath,amssymb,amsthm,mathtools}
\usepackage{bm}

\usepackage{graphicx}
\graphicspath{{figures/}}
\usepackage{booktabs}
\usepackage{array}

\usepackage[shortlabels]{enumitem}
\usepackage[numbers,sort&compress]{natbib}
\usepackage{xcolor}
\usepackage[colorlinks=true,linkcolor=blue,citecolor=blue,urlcolor=blue]{hyperref}
\usepackage[capitalise,noabbrev]{cleveref}

\theoremstyle{plain}
\newtheorem{theorem}{Theorem}[section]
\newtheorem{lemma}[theorem]{Lemma}
\newtheorem{proposition}[theorem]{Proposition}
\newtheorem{corollary}[theorem]{Corollary}

\theoremstyle{definition}

\newtheorem{example}[theorem]{Example}

\theoremstyle{remark}

\newcommand{\R}{\mathbb{R}}

\newcommand{\E}{\mathbb{E}}

\DeclareMathOperator{\spanop}{span}
\DeclareMathOperator{\ran}{ran}
\DeclareMathOperator{\diag}{diag}

\newcommand{\ip}[2]{\langle #1,#2\rangle}
\newcommand{\norm}[1]{\lVert #1\rVert}
\newcommand{\cn}{c_n}
\newcommand{\Qop}{\mathcal{Q}}
\newcommand{\one}{\mathbf{1}}
\newtheorem{conjecture}[theorem]{Conjecture}
\hypersetup{pdftitle={Sharpness and Rigidity of the Two-Factor Recht--Re Lower Bound}}

\title{Sharpness and Rigidity of the\\Two-Factor Recht--R\'e Lower Bound}
\author{Jinze Zhao\\University of California, San Diego\\\texttt{jiz419@ucsd.edu}}
\date{}
\begin{document}

\maketitle

\begin{abstract}
Products of positive semidefinite matrices arise when successive steps of a
randomized optimization method use different data points. Their average can
have a negative eigenvalue even though every individual matrix is positive
semidefinite. Lai and Lim proved at ICML 2020 that
\(\sum_{i\ne j}A_iA_j\succeq-\frac14 n(n-1)I\) whenever
\(A_i\succeq0\) and \(\sum_i A_i\preceq nI\), and conjectured that the
constant is sharp for every \(n\ge2\).
We give an explicit family of \(n\) rank-one matrices in dimension \(n\)
that attains this bound for every \(n\), proving that sharpness statement.
The construction compresses one direction of an orthonormal basis by
a factor \(1/\sqrt2\).
We also give a short, self-contained proof of the known inequality and
classify all equality cases: every extremizer contains a common invariant
\(n\)-dimensional block equivalent to the construction.
Consequently, dimension \(n\) is necessary for equality, and positive
definite matrices approach but never attain the lower endpoint.
Finally, strongly convex quadratic objectives realize the extremal negative
eigenvalue as a two-step expected gradient-descent update.

\end{abstract}


\section{Introduction}
\label{sec:introduction}
When a gradient method minimizes a quadratic objective, each step multiplies
the current error by a matrix. If different steps use different summands of a
finite-sum objective, the resulting matrix product depends on the order of
the selected indices. Even two steps already exhibit a phenomenon absent
from scalar arithmetic: an average of products of positive semidefinite
matrices can have a negative eigenvalue. Determining how negative that
eigenvalue can be is a precise extremal problem underlying this matrix
approach to sampling in optimization.

Recht and R\'e introduced a family of noncommutative
arithmetic--geometric mean questions at COLT 2012, motivated by
sampling with and without replacement in randomized optimization
\citep{recht2012}. Lai and Lim subsequently disproved the general
conjecture and obtained a sharper two-factor inequality
\citep[Theorem~4]{lai2020}.
Under the normalization \(A_1+\cdots+A_n\preceq nI\), their lower bound is
\(-n(n-1)I/4\). After their theorem they supplied an equality example for
\(n=2\) and conjectured sharpness for every \(n\ge2\).
That specific sharpness question is the subject of this paper.
It is stated on page~5616 of the ICML proceedings article, and appears
again in Lai's June 2023 dissertation \citep[p.~26]{lai2023}.
Thus the question dates to 2020, with an explicit later record in 2023.

The answer has a simple form. Let \(e_1,\ldots,e_n\) be the standard basis
of \(\R^n\). Write \(e=n^{-1/2}(1,\ldots,1)^T\), and set
\[
 P=ee^T,\qquad C=I-\left(1-\frac1{\sqrt2}\right)P,\qquad
 b_i=Ce_i,\qquad A_i=n\,b_i b_i^T.
\]
The operator \(C\) fixes every vector perpendicular to \(e\) and multiplies
\(e\) by \(1/\sqrt2\). Consequently
\[
 \sum_i A_i=nI-\frac n2P.
\]
The vectors \(b_i\) have equal squared length \(1-1/(2n)\).
This reduces the entire product sum to a polynomial in \(\sum_i A_i\),
and gives
\[
 \sum_{i\ne j}A_iA_j
   =\frac n2(I-P)-\frac{n(n-1)}4P.
\]
The claimed lower endpoint is therefore attained exactly on \(e\), for
every number \(n\) of matrices.

Our second result describes why this construction works.
Equality forces an \(n\)-dimensional subspace on which the matrices have
exactly the preceding form, up to an orthogonal change of coordinates.
The rest of the space may carry any feasible family of matrices.
This identifies the smallest possible dimension of an extremizer and
shows that equality requires singular matrices.
The proof uses an explicit embedding of positive semidefinite matrices
as compressions of coordinate projections; all details are included.

The paper gives the equality construction, a complete description of
equality cases, and a concrete quadratic-optimization interpretation.
The universal lower inequality itself is due to Lai and Lim; we provide
an elementary proof to make the argument self-contained.
The mathematical scope is the two-factor lower endpoint and its
extremizers. Higher product lengths and algorithmic convergence rates
require separate analyses.



\section{The extremal problem}
\label{sec:problem}

All matrices are real and finite-dimensional. For a symmetric matrix \(H\),
the notation \(H\succeq0\) means \(x^THx\ge0\) for every vector \(x\).
The relation \(H\preceq K\) means \(K-H\succeq0\).
The identity matrix on the space in use is denoted by \(I\).
Vector norms are Euclidean, matrix norms are the induced operator norms,
and \(\lambda_{\min}(H)\) denotes the smallest eigenvalue of a symmetric
matrix.

Fix integers \(n\ge2\) and \(d\ge1\).
We work first with matrices \(B_1,\ldots,B_n\in\R^{d\times d}\) satisfying
\begin{equation}
 B_i\succeq0,\qquad S:=\sum_{i=1}^n B_i\preceq I.
 \label{eq:feasible}
\end{equation}
Define the ordered product sum and the proposed lower constant by
\begin{equation}
 \Qop(B):=\sum_{\substack{1\le i,j\le n\\i\ne j}}B_iB_j,
 \qquad
 \cn:=\frac{n-1}{4n}.
 \label{eq:Q}
\end{equation}
There are \(n(n-1)\) terms in this sum. Reversing the two indices gives
\(\Qop(B)^T=\Qop(B)\), so its eigenvalues are real.
Expanding \(S^2\), with the order of multiplication preserved, gives
\begin{equation}
 \Qop(B)=S^2-\sum_{i=1}^n B_i^2.
 \label{eq:expand}
\end{equation}

The change of variables \(A_i=nB_i\) takes
\eqref{eq:feasible} to the normalization in Lai and Lim's theorem.
Under this change,
\[
 \sum_i A_i\preceq nI,\qquad
 \sum_{i\ne j}A_iA_j=n^2\Qop(B),\qquad
 n^2\cn=\frac{n(n-1)}4.
\]
Thus the sharpness question can be stated without any ambiguity in scaling.

\begin{conjecture}[Lai--Lim sharpness question, normalized]
\label{conj:sharp}
For every integer \(n\ge2\), there exist a dimension \(d\) and matrices
satisfying \eqref{eq:feasible} for which
\(\lambda_{\min}(\Qop(B))=-\cn\).
\end{conjecture}

This formulation paraphrases the sharpness assertion immediately following
Theorem~4 in \citet{lai2020}. The dimension is allowed to depend on \(n\).
No assertion of equality in every fixed dimension is part of this question.

\begin{theorem}[Sharp lower endpoint]
\label{thm:main}
For every \(n\ge2\), every \(d\ge1\), and every feasible family,
\begin{equation}
 -\cn I\preceq\Qop(B)\preceq\frac{n-1}{n}I.
 \label{eq:both}
\end{equation}
There is an explicit feasible family of \(n\) rank-one matrices in
dimension \(n\) for which
\begin{equation}
 \Qop(B)=\frac1{2n}(I-P)-\cn P,
 \label{eq:extremalQ}
\end{equation}
where \(P\) is a rank-one orthogonal projection.
In particular,
\begin{equation}
 \inf_{\substack{d\ge1,\ B_1,\ldots,B_n\succeq0\\
                  \sum_i B_i\preceq I}}
       \lambda_{\min}(\Qop(B))
       =-\frac{n-1}{4n},
 \label{eq:optimum}
\end{equation}
and the infimum is attained in dimension \(n\).
\end{theorem}

The two inequalities in \eqref{eq:both} are known results of
\citet{lai2020}. Their proof is included next.
The equality construction in Section~\ref{sec:construction} proves
Conjecture~\ref{conj:sharp}.
The equality classification in Section~\ref{sec:rigidity} will show that
dimension \(n\) is also necessary.

\section{A self-contained proof of the universal bounds}
\label{sec:bound}

The main elementary fact is that the difference of two positive
semidefinite parts of a contraction has a controlled square.
No commutativity is required.

\begin{lemma}[A difference-square inequality]
\label{lem:diff}
If \(E,F\succeq0\) and \(E+F\preceq I\), then
\begin{equation}
 (E-F)^2\preceq E+F.
 \label{eq:diff}
\end{equation}
\end{lemma}
\begin{proof}
Put \(T=E+F\) and \(D=E-F\).
Both \(T+D=2E\) and \(T-D=2F\) are positive semidefinite; hence
\(-T\preceq D\preceq T\).

We first justify using an inverse when \(T\) might be singular.
If \(z\in\ker T\), then
\(0=z^TTz=z^TEz+z^TFz\).
Each term is nonnegative, so both vanish.
For a positive semidefinite matrix \(E\),
\(z^TEz=\norm{E^{1/2}z}^2=0\) implies \(Ez=0\).
The same argument applies to \(F\).
Consequently \(Dz=0\).
Since \(D\) is symmetric, the orthogonal complement of \(\ker T\)
is invariant under \(D\).

Restrict now to that orthogonal complement.
The restriction of \(T\) is positive definite, and the symmetric matrix
\[
 K=T^{-1/2}DT^{-1/2}
\]
satisfies \(-I\preceq K\preceq I\) by congruence of the previous order
inequalities. Its eigenvalues lie in \([-1,1]\), so \(K^2\preceq I\).
Because \(T\preceq I\), congruence by \(K\) gives
\(KTK\preceq K^2\). Therefore
\[
 D^2
 =T^{1/2}KTKT^{1/2}
 \preceq T^{1/2}K^2T^{1/2}
 \preceq T.
\]
Both sides vanish on \(\ker T\).
The inequality therefore holds on the original space.
\end{proof}

For a feasible family, define for every \(i<j\)
\begin{equation}
 D_{ij}:=B_i+B_j-(B_i-B_j)^2.
 \label{eq:Dij}
\end{equation}
Since \(B_i+B_j\preceq S\preceq I\),
Lemma~\ref{lem:diff} gives \(D_{ij}\succeq0\).
A direct expansion supplies an exact positive decomposition of the
slack in the lower bound.

\begin{lemma}[The lower-bound slack]
\label{lem:slack}
For every feasible family,
\begin{equation}
 \Qop(B)+\cn I
   =\frac{n-1}{n}\left(S-\frac12I\right)^2
     +\frac1n\sum_{i<j}D_{ij}.
 \label{eq:slack}
\end{equation}
\end{lemma}
\begin{proof}
Each \(B_i^2\) occurs in exactly \(n-1\) of the pairwise squares.
Their mixed terms contain each ordered product \(B_iB_j\), \(i\ne j\),
once with a minus sign. Hence
\[
 \sum_{i<j}(B_i-B_j)^2
   =(n-1)\sum_i B_i^2-\Qop(B)
   =n\sum_i B_i^2-S^2,
\]
where the last equality uses \eqref{eq:expand}.
Each \(B_i\) likewise occurs in \(n-1\) of the sums \(B_i+B_j\).
Thus
\[
 \sum_{i<j}D_{ij}=(n-1)S-n\sum_i B_i^2+S^2.
\]
Substituting this into the right-hand side of \eqref{eq:slack} gives
\begin{align*}
 &\frac{n-1}{n}\left(S^2-S+\frac14I\right)
 +\frac1n\left((n-1)S-n\sum_i B_i^2+S^2\right)\\
 &\hspace{2em}=S^2-\sum_i B_i^2+\frac{n-1}{4n}I
 =\Qop(B)+\cn I.
\end{align*}
This proves the identity.
\end{proof}

\begin{proof}[Proof of the bounds in Theorem~\ref{thm:main}]
Every term on the right-hand side of \eqref{eq:slack} is positive
semidefinite. This proves the lower bound.

For the upper bound, the pairwise-square expansion also gives
\[
 n\sum_i B_i^2-S^2=\sum_{i<j}(B_i-B_j)^2\succeq0.
\]
Therefore
\[
 \Qop(B)=S^2-\sum_i B_i^2
 \preceq\frac{n-1}{n}S^2
 \preceq\frac{n-1}{n}I.
\]
The last step follows from the eigenvalues of \(S\) lying in \([0,1]\).
The upper constant is attained, for example, by \(B_i=I/n\) for all \(i\).
\end{proof}

The decomposition also records exactly what happens on an extremal vector.

\begin{corollary}[Vector equality conditions]
\label{cor:vector}
For a feasible family and a unit vector \(x\), the following are equivalent:
\begin{enumerate}[(i)]
 \item \(\Qop(B)x=-\cn x\);
 \item \(Sx=x/2\) and \(D_{ij}x=0\) for every \(i<j\).
\end{enumerate}
\end{corollary}
\begin{proof}
If (i) holds, taking the quadratic form of \eqref{eq:slack} at \(x\) gives
\[
 0=\frac{n-1}{n}\norm{(S-I/2)x}^2
      +\frac1n\sum_{i<j}\ip{x}{D_{ij}x}.
\]
All summands are nonnegative. Thus the norm is zero, and every quadratic
form in the sum is zero.
For a positive semidefinite \(D_{ij}\), this implies
\(D_{ij}x=0\), by the square-root argument used in
Lemma~\ref{lem:diff}.
Conversely, (ii) makes every operator on the right of \eqref{eq:slack}
annihilate \(x\), proving (i).
\end{proof}

\section{An equality construction for every number of matrices}
\label{sec:construction}

For the remainder of this section the dimension is \(d=n\).
Let
\begin{equation}
 e=\frac1{\sqrt n}\one,\quad P=ee^T,\quad
 a=1-\frac1{\sqrt2},\quad C=I-aP,\quad
 b_i=Ce_i,\quad B_i^\star=b_i b_i^T.
 \label{eq:canonical}
\end{equation}
The symbol \(e_i\) denotes the \(i\)th standard basis vector, while \(e\)
is the distinguished unit vector.
The projection relations \(P^2=P=P^T\) will make all computations explicit.

\begin{proposition}[Canonical extremizers]
\label{prop:canonical}
The matrices in \eqref{eq:canonical} are rank-one positive semidefinite,
and satisfy
\begin{align}
 S^\star:=\sum_i B_i^\star&=I-\frac12P,\label{eq:canonicalS}\\
 \ip{b_i}{b_j}&=\delta_{ij}-\frac1{2n},\label{eq:bgram}\\
 \Qop(B^\star)&=\frac1{2n}(I-P)-\cn P.\label{eq:canonicalQ}
\end{align}
In particular, they are feasible and attain the lower bound on \(e\).
\end{proposition}
\begin{proof}
For any vector \(z\),
\(z^TB_i^\star z=(b_i^Tz)^2\ge0\).
The matrix \(C\) has eigenvalue \(1/\sqrt2\) on \(\spanop\{e\}\)
and eigenvalue \(1\) on \(e^\perp\), so it is invertible.
Consequently every \(b_i\) is nonzero and \(B_i^\star\) has rank one.

By \(P^2=P\),
\[
 C^2=I-(2a-a^2)P=I-\frac12P,
\]
because \(2a-a^2=1-(1-a)^2=1-1/2=1/2\).
Using \(\sum_i e_i e_i^T=I\), we obtain
\[
 S^\star=\sum_i Ce_i e_i^TC=C^2=I-\frac12P.
\]
Its eigenvalues are \(1/2\) on \(e\) and \(1\) on \(e^\perp\);
hence \(0\preceq S^\star\preceq I\).

Every entry of \(P\) is \(1/n\). It follows that
\[
 \ip{b_i}{b_j}=e_i^TC^2e_j
 =\delta_{ij}-\frac1{2n}.
\]
In particular, \(\norm{b_i}^2=1-1/(2n)\) for every \(i\).
The rank-one multiplication rule now gives
\[
 (B_i^\star)^2
 =(b_i b_i^T)(b_i b_i^T)
 =\norm{b_i}^2\,b_i b_i^T
 =\left(1-\frac1{2n}\right)B_i^\star.
\]
Sum this equality and use \eqref{eq:expand}:
\[
 \Qop(B^\star)=(S^\star)^2
       -\left(1-\frac1{2n}\right)S^\star.
\]
On \(e^\perp\), this polynomial has value \(1/(2n)\).
On \(e\), it has value
\[
 \frac14-\frac12\left(1-\frac1{2n}\right)
 =-\frac14+\frac1{4n}
 =-\frac{n-1}{4n}.
\]
These two invariant spaces give exactly \eqref{eq:canonicalQ}.
\end{proof}

\begin{proof}[Completion of Theorem~\ref{thm:main}]
Proposition~\ref{prop:canonical} proves attainment for every \(n\).
Combined with the universal lower bound, it proves \eqref{eq:optimum}
and Conjecture~\ref{conj:sharp}.
Rescaling by \(A_i=nB_i^\star\) gives
\[
 \sum_i A_i=nI-\frac n2P,\qquad
 \sum_{i\ne j}A_iA_j=\frac n2(I-P)-\frac{n(n-1)}4P.
\]
Thus the lower constant in the original normalization is attained exactly.
\end{proof}

\begin{example}[Two matrices]
\label{ex:two}
In a basis whose first vector is the distinguished extremal direction,
the original normalization admits the particularly simple example
\[
 A_1=
 \begin{pmatrix}
  1/2&1/\sqrt2\\1/\sqrt2&1
 \end{pmatrix},
 \qquad
 A_2=
 \begin{pmatrix}
  1/2&-1/\sqrt2\\-1/\sqrt2&1
 \end{pmatrix}.
\]
These are outer products of \((1/\sqrt2,1)^T\) and
\((1/\sqrt2,-1)^T\), respectively. Consequently they are positive
semidefinite. Their sum is \(\diag(1,2)\preceq2I\).
Their squares are \(3A_1/2\) and \(3A_2/2\), respectively, so
\[
 A_1A_2+A_2A_1
 =\diag(1,4)-\frac32\diag(1,2)
 =\diag(-1/2,1).
\]
This illustrates the negative endpoint \(-2(2-1)/4=-1/2\).
\end{example}

\section{All equality cases and the minimum dimension}
\label{sec:rigidity}

A subspace \(K\subseteq\R^d\) is \emph{common reducing} for the family
\(B_1,\ldots,B_n\) if both \(K\) and \(K^\perp\) are invariant under every
\(B_i\). Since the matrices are symmetric, invariance of \(K\) alone
implies invariance of \(K^\perp\):
for \(z\in K^\perp\) and \(y\in K\),
\(\ip{B_i z}{y}=\ip{z}{B_i y}=0\).

We will prove a complete equality characterization.
Here \(B_i^\star\) always denotes the fixed canonical tuple
\eqref{eq:canonical} in \(\R^n\).

\begin{theorem}[Rigidity of the lower endpoint]
\label{thm:rigidity}
Let \(B_1,\ldots,B_n\) satisfy \eqref{eq:feasible}.
Then \(\lambda_{\min}(\Qop(B))=-\cn\) if and only if there exist
a common reducing subspace \(K\) of dimension \(n\) and an isometry
\(U:\R^n\to K\) such that
\begin{equation}
 B_i|_K=U B_i^\star U^T,\qquad 1\le i\le n.
 \label{eq:block}
\end{equation}
Equivalently, after an orthogonal change of coordinates,
\begin{equation}
 B_i=B_i^\star\oplus R_i,\qquad
 R_i\succeq0,\quad\sum_i R_i\preceq I_{d-n}.
 \label{eq:directsum}
\end{equation}
When \(d=n\), the residual block is absent.
In particular, equality is possible only if \(d\ge n\).
\end{theorem}

We begin with an explicit construction of the projection representation
used in the proof.

\begin{lemma}[An isometric projection representation]
\label{lem:embedding}
Given a feasible family, put \(B_0=I-S\) and define
\[
 V:\R^d\longrightarrow \mathcal H:=(\R^d)^{n+1},\qquad
 Vz=(B_0^{1/2}z,B_1^{1/2}z,\ldots,B_n^{1/2}z).
\]
Let \(P_i\) denote the orthogonal projection in \(\mathcal H\) onto its
\(i\)th coordinate block, and put \(W=\ran V\), \(R=VV^T\).
Then \(V^TV=I\), \(R\) is the orthogonal projection onto \(W\), and
\[
 V^TP_iV=B_i\quad(0\le i\le n).
\]
For \(i\ne j\) in \(\{1,\ldots,n\}\), the slack matrix in
\eqref{eq:Dij} satisfies
\begin{equation}
 D_{ij}=V^T(P_i-P_j)(I-R)(P_i-P_j)V.
 \label{eq:projectiongap}
\end{equation}
Consequently, \(D_{ij}x=0\) implies \((P_i-P_j)Vx\in W\).
\end{lemma}
\begin{proof}
Every \(B_i\), including \(B_0\), is positive semidefinite, so its
symmetric positive semidefinite square root is defined by the spectral
decomposition. The definition gives
\[
 V^TV=\sum_{i=0}^n B_i=I.
\]
Thus \(V\) preserves inner products and is injective.
Moreover, \(R^T=R\) and
\(R^2=V(V^TV)V^T=VV^T=R\), with range \(W\);
this proves that \(R\) is the orthogonal projection onto \(W\).
The identity \(V^TP_iV=B_i\) follows by selecting the \(i\)th coordinate
in the displayed definition of \(V\).

Put \(T=P_i-P_j\).
Distinct coordinate projections are orthogonal, so
\(T^2=P_i+P_j\).
Therefore
\begin{align*}
 V^TT(I-R)TV
 &=V^TT^2V-V^TT(VV^T)TV\\
 &=B_i+B_j-(V^TTV)^2\\
 &=B_i+B_j-(B_i-B_j)^2
 =D_{ij}.
\end{align*}
For any \(x\),
\[
 \ip{x}{D_{ij}x}=\norm{(I-R)(P_i-P_j)Vx}^2.
\]
If \(D_{ij}x=0\), this norm vanishes.
Hence \((P_i-P_j)Vx=R(P_i-P_j)Vx\in W\).
\end{proof}

\begin{proof}[Proof of Theorem~\ref{thm:rigidity}, necessity]
Suppose that \(\lambda_{\min}(\Qop(B))=-\cn\).
There is a unit eigenvector \(x\) at that eigenvalue because
\(\Qop(B)\) is a real symmetric matrix.
By Corollary~\ref{cor:vector},
\begin{equation}
 Sx=\frac12x,\qquad D_{ij}x=0\quad(i<j).
 \label{eq:xequality}
\end{equation}
Use the notation of Lemma~\ref{lem:embedding}, and set
\[
 v=Vx,\qquad z_i=P_i v\ (0\le i\le n),\qquad
 z=\sum_{i=1}^n z_i.
\]
The vectors \(z_0,\ldots,z_n\) are pairwise orthogonal because they
belong to distinct coordinate blocks.
They satisfy \(v=z_0+z\).
Furthermore,
\[
 \norm{v}^2=1,\qquad
 \norm{z}^2=\ip{x}{Sx}=\frac12,\qquad
 \norm{z_0}^2=\ip{x}{(I-S)x}=\frac12.
\]
The projection of \(z\) onto \(W\) is
\[
 Rz=VV^T\left(\sum_{i=1}^n P_i\right)Vx
    =VSx=\frac12v.
\]
Define
\begin{equation}
 r:=z-\frac12v=\frac12(z-z_0).
 \label{eq:r}
\end{equation}
Then \(r=(I-R)z\in W^\perp\), and orthogonality of \(z\) and \(z_0\)
gives
\begin{equation}
 \norm{r}^2=\frac14(\norm{z}^2+\norm{z_0}^2)=\frac14.
 \label{eq:rnorm}
\end{equation}

For every \(i\ne j\) with \(i,j\ge1\),
Lemma~\ref{lem:embedding} and \eqref{eq:xequality} imply
\begin{equation}
 z_i-z_j\in W.
 \label{eq:differencesW}
\end{equation}
Take its inner product with \(r\in W^\perp\).
Since the \(z_i\)'s are pairwise orthogonal,
\[
 0=\ip{r}{z_i-z_j}
   =\frac12\bigl(\norm{z_i}^2-\norm{z_j}^2\bigr).
\]
Thus all \(\norm{z_i}^2\), for \(1\le i\le n\), are equal.
Their sum is \(\norm{z}^2=1/2\), so
\begin{equation}
 \norm{z_i}^2=\frac1{2n}\qquad(1\le i\le n).
 \label{eq:ziweights}
\end{equation}

The differences in \eqref{eq:differencesW} also imply that
\((I-R)z_i\) is the same vector for every \(i\ge1\).
Their sum is \((I-R)z=r\), and consequently
\[
 (I-R)z_i=\frac rn,\qquad
 Rz_i=z_i-\frac rn.
\]
Let
\begin{equation}
 w_i=B_i x\in\R^d.
 \label{eq:wi}
\end{equation}
Since \(V^Tz_i=B_i x\), we have
\[
 Vw_i=Rz_i=z_i-\frac rn.
\]
We next compute the Gram matrix of the \(w_i\)'s.
Equation~\eqref{eq:r} and orthogonality give
\(\ip{z_i}{r}=\norm{z_i}^2/2=1/(4n)\).
As \(V\) is an isometry, \eqref{eq:rnorm} yields
\begin{align}
 \ip{w_i}{w_j}
 &=\ip{z_i-r/n}{z_j-r/n}\notag\\
 &=\frac{\delta_{ij}}{2n}
       -\frac1n\frac1{4n}
       -\frac1n\frac1{4n}
       +\frac1{n^2}\frac14\notag\\
 &=\frac{\delta_{ij}}{2n}-\frac1{4n^2}.
 \label{eq:wgram}
\end{align}
This \(n\)-by-\(n\) Gram matrix is
\[
 G=\frac1{2n}I_n-\frac1{4n^2}\one\one^T.
\]
On \(\one^\perp\) it has eigenvalue \(1/(2n)\), and on
\(\spanop\{\one\}\) it has eigenvalue
\(1/(2n)-n/(4n^2)=1/(4n)\).
Both are positive.
The vectors \(w_1,\ldots,w_n\) are therefore linearly independent.
In particular,
\begin{equation}
 K:=\spanop\{w_1,\ldots,w_n\}\quad\hbox{has dimension }n,
 \qquad d\ge n.
 \label{eq:K}
\end{equation}
Also \(\sum_i w_i=Sx=x/2\), so \(x\in K\).

It remains to determine how each \(B_i\) acts on this subspace.
For distinct \(j,k\), equations \eqref{eq:differencesW} and
\eqref{eq:wi} imply
\[
 V(w_j-w_k)=R(z_j-z_k)=z_j-z_k.
\]
Using \(B_i=V^TP_iV\), we obtain, for every \(i\),
\begin{align}
 B_i(w_j-w_k)
 &=V^TP_i(z_j-z_k)\notag\\
 &=\delta_{ij}V^Tz_i-\delta_{ik}V^Tz_i
 =(\delta_{ij}-\delta_{ik})w_i.
 \label{eq:actiondifference}
\end{align}
Thus for fixed \(i\), the vector \(B_iw_j-\delta_{ij}w_i\)
does not depend on \(j\). Denote it by \(t_i\).
Summing over \(j\), and using \(\sum_jw_j=x/2\), gives
\[
 \frac12w_i
 =B_i\left(\sum_j w_j\right)
 =\sum_j(\delta_{ij}w_i+t_i)
 =w_i+nt_i.
\]
It follows that \(t_i=-w_i/(2n)\), or
\begin{equation}
 B_iw_j=\left(\delta_{ij}-\frac1{2n}\right)w_i
 \quad\hbox{for every }i,j.
 \label{eq:action}
\end{equation}
In particular, \(K\) is invariant under every \(B_i\).
Symmetry then makes it a common reducing subspace.

Equation~\eqref{eq:wgram} shows that the rank-one operator
\(2n w_i w_i^T\) has exactly the same action on each basis vector \(w_j\):
\[
 2n w_iw_i^Tw_j
 =2n\left(\frac{\delta_{ij}}{2n}-\frac1{4n^2}\right)w_i
 =\left(\delta_{ij}-\frac1{2n}\right)w_i.
\]
Consequently,
\begin{equation}
 B_i|_K=2n w_iw_i^T|_K.
 \label{eq:rankoneblock}
\end{equation}

To identify this block with the canonical construction, put
\(w_i^\star=B_i^\star e\).
Since \(Ce=e/\sqrt2\),
\[
 \ip{b_i}{e}=e_i^TCe=\frac1{\sqrt{2n}},
 \qquad w_i^\star=\frac{b_i}{\sqrt{2n}}.
\]
Equation~\eqref{eq:bgram} gives
\[
 \ip{w_i^\star}{w_j^\star}
   =\frac{\delta_{ij}}{2n}-\frac1{4n^2}
   =\ip{w_i}{w_j}.
\]
Both sets are linearly independent.
Define \(U\) by \(Uw_i^\star=w_i\), extending linearly.
Equality of the two Gram matrices shows explicitly that for any real
coefficients \(\alpha_i,\beta_j\),
\[
 \ip{U\sum_i\alpha_iw_i^\star}{U\sum_j\beta_jw_j^\star}
 =\sum_{i,j}\alpha_i\beta_j\ip{w_i}{w_j}
 =\ip{\sum_i\alpha_iw_i^\star}{\sum_j\beta_jw_j^\star}.
\]
Thus \(U:\R^n\to K\) is an isometry.
Since \(B_i^\star=2n w_i^\star(w_i^\star)^T\),
\eqref{eq:rankoneblock} becomes
\(B_i|_K=UB_i^\star U^T\).
This proves \eqref{eq:block}.

On \(K^\perp\), let \(R_i=B_i|_{K^\perp}\).
Restriction of the feasibility inequalities gives
\(R_i\succeq0\) and \(\sum_iR_i\preceq I_{K^\perp}\).
This yields \eqref{eq:directsum}.
\end{proof}

\begin{proof}[Proof of Theorem~\ref{thm:rigidity}, sufficiency]
Suppose \eqref{eq:directsum} holds after an orthogonal change of coordinates.
Multiplication and addition preserve the block decomposition, so
\[
 \Qop(B)=\Qop(B^\star)\oplus\Qop(R).
\]
By Proposition~\ref{prop:canonical}, the first block has eigenvalue
\(-\cn\).
The residual family is feasible, so Theorem~\ref{thm:main} bounds every
eigenvalue of its product sum below by \(-\cn\).
Therefore the smallest eigenvalue of the direct sum is exactly \(-\cn\).
Orthogonal changes of coordinates preserve eigenvalues, completing the
proof.
\end{proof}

\section{Consequences of rigidity}
\label{sec:consequences}

\begin{corollary}[Strict improvement below dimension \(n\)]
\label{cor:dimension}
For fixed \(n\ge2\) and \(1\le d<n\), there exists
\(\varepsilon_{n,d}>0\) such that every feasible \(d\)-dimensional family
satisfies
\[
 \Qop(B)\succeq(-\cn+\varepsilon_{n,d})I.
\]
\end{corollary}
\begin{proof}
The feasible set is closed: positive semidefiniteness and the inequality
\(\sum_iB_i\preceq I\) are preserved under entrywise limits.
It is bounded because \(0\preceq B_i\preceq I\) for every \(i\), which
bounds all matrix entries.
In finite dimension the feasible set is therefore compact.
The function \(B\mapsto\lambda_{\min}(\Qop(B))\) is continuous.
Indeed, \(B\mapsto\Qop(B)\) is polynomial, and for symmetric \(H,K\)
the variational formula
\(\lambda_{\min}(H)=\min_{\norm{x}=1}x^THx\) gives
\(\lvert\lambda_{\min}(H)-\lambda_{\min}(K)\rvert\le\norm{H-K}\).
Thus the minimum over feasible families is attained.
It is at least \(-\cn\), and cannot equal \(-\cn\) by
Theorem~\ref{thm:rigidity}, since \(d<n\).
The difference between the attained minimum and \(-\cn\) is the required
positive \(\varepsilon_{n,d}\).
\end{proof}

This corollary asserts existence of a gap for each fixed \(n,d\);
it does not supply a formula for that gap.

\begin{corollary}[Positive definite matrices]
\label{cor:pd}
If all \(B_i\) are positive definite, then
\(\lambda_{\min}(\Qop(B))>-\cn\).
Nevertheless, the infimum over positive definite feasible families,
allowing dimension \(n\), is still \(-\cn\).
\end{corollary}
\begin{proof}
If equality held, Theorem~\ref{thm:rigidity} would give an
\(n\)-dimensional reducing subspace on which every \(B_i\) has rank one.
As \(n\ge2\), each such restriction has a nonzero kernel vector.
That vector would also lie in the kernel of \(B_i\) on the full space,
contradicting positive definiteness.

For approximation, take the canonical family and, for \(0<t<1\), define
\[
 B_i(t)=(1-t)B_i^\star+\frac tn I_n.
\]
For any nonzero \(z\),
\(z^TB_i(t)z\ge(t/n)\norm{z}^2>0\), so the matrices are positive definite.
Their sum is
\[
 \sum_i B_i(t)=(1-t)S^\star+tI_n\preceq I_n,
\]
so they are feasible.
The product sum has the exact expansion
\begin{equation}
 \Qop(B(t))
  =(1-t)^2\Qop(B^\star)
   +\frac{2t(1-t)(n-1)}n S^\star
   +\frac{t^2(n-1)}n I_n.
 \label{eq:pdexpansion}
\end{equation}
To verify it, expand each ordered product:
the first terms sum to \((1-t)^2\Qop(B^\star)\);
the mixed terms contain each \(B_i^\star\) exactly \(2(n-1)\) times
with coefficient \(t(1-t)/n\);
and the \(n(n-1)\) constant terms have coefficient \(t^2/n^2\).

In particular, \(e\) is an eigenvector in \eqref{eq:pdexpansion},
with eigenvalue
\[
 -(1-t)^2\cn+\frac{t(1-t)(n-1)}n+\frac{t^2(n-1)}n
 =-(1-t)^2\cn+\frac{t(n-1)}n,
\]
which tends to \(-\cn\) as \(t\downarrow0\).
The universal lower bound gives the reverse bound on the infimum,
proving the claim.
\end{proof}

\begin{proposition}[Commuting families have no negative endpoint]
\label{prop:commuting}
If a feasible family commutes pairwise, then \(\Qop(B)\succeq0\).
The canonical family is pairwise noncommuting when \(n\ge2\).
\end{proposition}
\begin{proof}
Commuting real symmetric matrices admit a common orthonormal eigenbasis.
For completeness, decompose into eigenspaces of \(B_1\).
Every \(B_j\) preserves those spaces because it commutes with \(B_1\).
Diagonalize \(B_2\) within them, and continue over the finite family.
The resulting common basis diagonalizes every \(B_i\).
Every diagonal entry is nonnegative by positive semidefiniteness, so
every diagonal entry of \(\Qop(B)\) is a sum of nonnegative scalar
products. This proves the first assertion.

For the canonical family and \(i\ne j\),
\[
 B_i^\star B_j^\star
 =-\frac1{2n}b_i b_j^T,\qquad
 B_j^\star B_i^\star
 =-\frac1{2n}b_j b_i^T.
\]
The vectors \(b_i,b_j\) are linearly independent since \(C\) is invertible.
Hence \(b_i b_j^T\ne b_j b_i^T\):
otherwise applying the equality to \(b_i\) would make \(b_j\)
a multiple of \(b_i\).
The two products are therefore different.
\end{proof}

\section{A two-step quadratic optimization realization}
\label{sec:optimization}

This section connects the endpoint to an explicit finite-sum optimization
problem. It also fixes the quantity being compared: the expected iterate
after two steps.

\begin{proposition}[Strongly convex quadratics realizing the endpoint]
\label{prop:optimization}
Let \(B_i^\star\) be the canonical family in dimension \(n\), and define
\[
 H_i=I-B_i^\star,\qquad
 f_i(x)=\frac12x^TH_i x,\qquad
 f(x)=\frac1n\sum_{i=1}^n f_i(x).
\]
Every \(f_i\) has gradient Lipschitz constant \(1\), is strongly convex
with parameter \(1/(2n)\), and has unique minimizer \(0\).
Start at \(x_0=e\), and take two component-gradient steps with stepsize \(1\):
\[
 x_1=x_0-\nabla f_i(x_0),\qquad
 x_2=x_1-\nabla f_j(x_1).
\]
If \((i,j)\) is uniform over ordered distinct pairs, then
\begin{equation}
 \E_{\rm distinct}x_2=-\frac1{4n^2}e.
 \label{eq:distinctiterate}
\end{equation}
If \(i,j\) are independent and uniform, then
\begin{equation}
 \E_{\rm independent}x_2=\frac1{4n^2}e.
 \label{eq:independentiterate}
\end{equation}
The distinct-pair mean update matrix has the smallest possible
eigenvalue among all positive semidefinite update families with
\(\sum_iB_i\preceq I\).
\end{proposition}
\begin{proof}
A rank-one matrix \(B_i^\star=b_i b_i^T\) has eigenvalue
\(\norm{b_i}^2=1-1/(2n)\) on \(b_i\), and eigenvalue zero
on \(b_i^\perp\).
Thus \(H_i\) has eigenvalue \(1/(2n)\) on \(b_i\) and eigenvalue \(1\)
on \(b_i^\perp\).
For every \(y\),
\[
 \frac1{2n}\norm{y}^2\le y^TH_i y\le\norm{y}^2.
\]
Since \(\nabla f_i(x)=H_i x\), the spectral upper bound gives
\(\norm{\nabla f_i(x)-\nabla f_i(z)}\le\norm{x-z}\).
The exact quadratic identity
\[
 f_i(z)-f_i(x)-\nabla f_i(x)^T(z-x)
 =\frac12(z-x)^TH_i(z-x)
 \ge\frac1{4n}\norm{z-x}^2
\]
proves strong convexity with parameter \(1/(2n)\).
Positive definiteness of \(H_i\) makes \(0\) its unique minimizer.
The same minimizer is shared by the average \(f\).

The gradient update is \(x\mapsto(I-H_i)x=B_i^\star x\).
Hence \(x_2=B_j^\star B_i^\star x_0\).
For uniform distinct indices, its mean update operator is
\begin{align}
 M_{\rm distinct}
 &=\frac1{n(n-1)}\Qop(B^\star)\notag\\
 &=\frac1{2n^2(n-1)}(I-P)-\frac1{4n^2}P.
 \label{eq:meanmatrix}
\end{align}
Applying it to \(e\), where \(Pe=e\), proves
\eqref{eq:distinctiterate}.
For independent indices, distributivity gives
\[
 M_{\rm independent}
 =\frac1{n^2}\sum_{i,j}B_j^\star B_i^\star
 =\frac1{n^2}(S^\star)^2
 =\frac1{n^2}(I-P)+\frac1{4n^2}P.
\]
This proves \eqref{eq:independentiterate}.

Finally, for any feasible update family,
Theorem~\ref{thm:main} gives
\[
 \frac1{n(n-1)}\Qop(B)\succeq
 -\frac{\cn}{n(n-1)}I=-\frac1{4n^2}I.
\]
Equation~\eqref{eq:meanmatrix} attains that lower eigenvalue,
which proves the extremal assertion.
\end{proof}

The negative eigenvalue records a reversal of direction in the mean
iterate. It is not a negative objective value: every \(f_i\) and \(f\)
is nonnegative. The distinction can also be checked directly on the
individual trajectories.

\begin{proposition}[Individual distinct-pair trajectories]
\label{prop:trajectory}
For the initialization and objectives in
Proposition~\ref{prop:optimization}, every ordered distinct pair satisfies
\[
 x_2=-\frac{b_j}{2n\sqrt{2n}},\qquad
 \norm{x_2}^2=\frac{2n-1}{16n^4}>0.
\]
In particular,
\[
 \E_{\rm distinct}\norm{x_2}^2
 -\norm{\E_{\rm distinct}x_2}^2
 =\frac{n-1}{8n^4}>0.
\]
\end{proposition}
\begin{proof}
From \(\ip{b_i}{e}=1/\sqrt{2n}\) we have
\(B_i^\star e=b_i/\sqrt{2n}\).
For \(i\ne j\), equation~\eqref{eq:bgram} gives
\(\ip{b_j}{b_i}=-1/(2n)\).
Therefore
\[
 B_j^\star B_i^\star e
 =\frac1{\sqrt{2n}}b_j(b_j^Tb_i)
 =-\frac{b_j}{2n\sqrt{2n}}.
\]
Its squared norm is
\[
 \frac{\norm{b_j}^2}{8n^3}
 =\frac{1-1/(2n)}{8n^3}
 =\frac{2n-1}{16n^4}.
\]
This value is independent of the distinct pair.
On the other hand, \eqref{eq:distinctiterate} has squared norm
\(1/(16n^4)\).
Subtracting gives
\((2n-2)/(16n^4)=(n-1)/(8n^4)\).
\end{proof}

These identities concern a specified two-step statistic on an explicit
family. They provide a direct optimization realization of the sharp
matrix endpoint. Expected objective values and behavior over additional
epochs are different functionals of the random products.

\section{Conclusion}
\label{sec:conclusion}
The explicit rank-one tuple in \eqref{eq:canonical} attains
\(-n(n-1)/4\) in the normalization
\(A_i\succeq0\), \(\sum_i A_i\preceq nI\), for every \(n\ge2\).
It therefore proves the sharpness statement posed after Theorem~4 of
\citet{lai2020}.
The equality classification shows that the same tuple is forced on
an \(n\)-dimensional common reducing subspace of every extremizer.
This establishes minimum dimension, uniqueness of the canonical block
up to orthogonal equivalence, and nonattainment by positive definite
matrices.

The fixed-dimension optimum below dimension \(n\) remains an additional
question: compactness proves a strict gap but does not determine its
value. A quantitative stability version of the rigidity theorem is
another natural direction, asking whether a near-extremal eigenvalue
forces an approximately canonical block.
The present results concern the exact two-factor endpoint; these
further questions require additional arguments.

\section{Disclosure}
\label{sec:disclosure}
The proof strategy is produced by OpenAI’s GPT-6 Astra Ultra through Codex in
response to prompts from the author. Codex was also used to revise
the exposition and prepare the LaTeX manuscript. The author selected the
problem, directed the interactions and revisions, and is the sole named author. The AI system is acknowledged as a reasoning and writing tool, not
as an author. This disclosure is not a substitute for independent expert
mathematical review.

\bibliographystyle{plainnat}
\bibliography{references}

\appendix
\section{The regular-simplex interpretation}
\label{sec:simplex}
The basis-compression formula has an equivalent geometric description.
For \(1\le i\le n\), define
\[
 q_i=e_i-\frac1{\sqrt n}e,\qquad
 u_i=\sqrt{\frac n{n-1}}\,q_i.
\]
Then \(e^Tq_i=0\), because
\(e^Te_i=1/\sqrt n\) and \(e^Te=1\).
Also
\[
 \sum_iq_i=\one-\sqrt n\,e=0,\qquad
 \ip{q_i}{q_j}=\delta_{ij}-\frac1n.
\]
Consequently,
\[
 \norm{u_i}=1,\qquad
 \ip{u_i}{u_j}=-\frac1{n-1}\quad(i\ne j).
\]
These are the vertices of a centered regular simplex in \(e^\perp\).
The outer-product sum is
\begin{align*}
 \sum_iq_iq_i^T
 &=\sum_i e_ie_i^T
   -\frac1{\sqrt n}\sum_i(e_i e^T+e e_i^T)
   +\sum_i\frac1n ee^T\\
 &=I-2P+P=I-P,
\end{align*}
and hence \(\sum_i u_i u_i^T=\frac n{n-1}(I-P)\).

The canonical vectors satisfy
\[
 b_i=e_i-\left(1-\frac1{\sqrt2}\right)\frac e{\sqrt n}
     =q_i+\frac e{\sqrt{2n}}.
\]
In the original normalization, \(v_i=\sqrt n\,b_i\) therefore becomes
\[
 v_i=\frac e{\sqrt2}+\sqrt{n-1}\,u_i,\qquad
 A_i=v_i v_i^T.
\]
The simplex identities yield
\begin{align*}
 \sum_i v_i v_i^T
 &=\frac n2P
   +\sqrt{\frac{n-1}{2}}
       \left(e\sum_i u_i^T+\left(\sum_i u_i\right)e^T\right)
   +(n-1)\sum_i u_i u_i^T\\
 &=\frac n2P+n(I-P).
\end{align*}
Moreover,
\(\norm{v_i}^2=1/2+(n-1)=n-1/2\).
Thus
\[
 \sum_i A_i^2=\left(n-\frac12\right)\sum_i A_i,
\]
recovering the product-sum formula by direct geometric computation.

\section{Exact verification and its scope}
\label{sec:verification}
The accompanying file \texttt{verify\_exact.py} checks the construction
using exact symbolic arithmetic. It has two parts.
First it verifies scalar coefficient identities with \(n\) left as
a symbolic positive integer.
Second it constructs the matrices explicitly for \(n=2,\ldots,8\),
forms all \(n(n-1)\) ordered products, and compares the result
entry by entry with the claimed formula.
The finite checks include the Gram matrix
\eqref{eq:wgram}, the action \eqref{eq:action}, and the slack
decomposition \eqref{eq:slack}.
All 15 symbolic identities and all 16 check groups for each tested
dimension passed.

The core algebra can also be checked with a single formal projection
\(P\), using only \(P^2=P\). For
\(a=1-1/\sqrt2\) and \(S=I-P/2\), the needed identities are
\[
 2a-a^2=\frac12,\qquad
 S^2=I-\frac34P,
\]
and
\begin{align*}
 S^2-\left(1-\frac1{2n}\right)S
 &=I-\frac34P
   -\left(1-\frac1{2n}\right)\left(I-\frac12P\right)\\
 &=\frac1{2n}I-\left(\frac14+\frac1{4n}\right)P\\
 &=\frac1{2n}(I-P)-\frac{n-1}{4n}P.
\end{align*}
This is an identity for arbitrary \(n\), not a numerical fit.

Positive semidefiniteness of the construction is certified structurally:
\(B_i^\star=b_i b_i^T\) is an outer product, and
\(I-\sum_iB_i^\star=P/2\) is a nonnegative multiple of a projection.
The verifier checks these identities exactly; it does not infer
positive semidefiniteness from approximate eigenvalues.

\begin{table}[htbp]
 \centering
 \caption{Exact endpoint values for representative values of \(n\).
 The second column is the normalized endpoint and the third uses
 \(A_i=nB_i\). The all-\(n\) result is proved in the main text.}
 \begin{tabular}{ccc}
  \toprule
  \(n\)&\(\lambda_{\min}(\Qop(B^\star))\)&
  \(\lambda_{\min}(\sum_{i\ne j}A_iA_j)\)\\
  \midrule
  2&\(-1/8\)&\(-1/2\)\\
  3&\(-1/6\)&\(-3/2\)\\
  4&\(-3/16\)&\(-3\)\\
  5&\(-1/5\)&\(-5\)\\
  6&\(-5/24\)&\(-15/2\)\\
  7&\(-3/14\)&\(-21/2\)\\
  8&\(-7/32\)&\(-14\)\\
  \bottomrule
 \end{tabular}
\end{table}

These checks are supplementary safeguards against transcription and
normalization errors.
The proofs for every \(n\), every feasible input family, and the
classification of arbitrary equality cases are the arguments in
Sections~\ref{sec:bound}--\ref{sec:rigidity}.
The computation is not a formal proof-assistant certificate and
does not verify bibliographic priority.

\end{document}